\documentclass[11pt]{article}

\usepackage{amsmath, amssymb ,amsthm, amsfonts, amsgen, color, verbatim}

\numberwithin{equation}{section}

\newcommand{\R}{{\mathbb{R}}}
\newcommand{\Z}{{\mathbb{Z}}}
\newcommand{\T}{{\mathbb{T}^N}}

\newcommand{\beq}{\begin{equation}}
\newcommand{\eeq}{\end{equation}}

\def\into{\rightarrow}

\newtheorem{Theorem}{Theorem}[section]
\newtheorem{Lemma}[Theorem]{Lemma}
\newtheorem{Proposition}[Theorem]{Proposition}

\newtheorem{Remark}[Theorem]{Remark}

\newcommand\keywordsname{Key words}
\newcommand\AMSname{AMS subject classifications}

\newcommand\nd{\noindent}
\newcommand\eps{\varepsilon}
\newcommand\f{\varphi}
\begin{document}

\author{
Radu Ignat
\setcounter{footnote}{1}
\footnote{Institut de Math\'ematiques de Toulouse \& Institut Universitaire de France, UMR 5219, Universit\'e de Toulouse, CNRS, UPS
IMT, F-31062 Toulouse Cedex 9, France. Email: Radu.Ignat@math.univ-toulouse.fr} 
\and Hamdi
Zorgati
\setcounter{footnote}{2}
\footnote{D\'epartement de Math\'ematiques, Facult\'e des Sciences de Tunis, Universit\'e Tunis El Manar 2092, Tunisia. 
Email: hamdi.zorgati@fst.rnu.tn}}

\title{Dimension reduction and optimality of the uniform state in a Phase-Field-Crystal model involving a higher order functional}

\maketitle
\begin{abstract}
We study a Phase-Field-Crystal model described by a free energy functional involving second order derivatives of the order parameter in a periodic setting and under a fixed mass constraint. We prove a $\Gamma$-convergence result in an asymptotic thin-film regime leading to a reduced $2$-dimensional model.
For the   reduced model, we prove necessary and sufficient conditions for the global minimality of the uniform state. We also prove similar results for the Ohta-Kawasaki model. 


\medskip

\end{abstract}

\section{Introduction and main results}

Recently Phase-Field-Crystal (PFC) models were introduced in order to study crystallization phenomena and to describe the pattern formation at microscopic scales.
These models succeed to capture the competition of attractive and repulsive interactions between some modulated phases inducing inhomogeneities and domain formation.
(We refer to the review paper of Emmerich et al. \cite{ELWGTTG} for more details.) 

In this paper, we consider a $3$-dimensional model inspired by the one derived by Elder et al. \cite{EKHG,EPBSG}. More precisely, this PFC model is
described by a free energy functional 
for the order parameter corresponding to the local mass density (or the number density of particles)
 which is a variant of the Swift-Hohenberg energy \cite{SH} (introduced to study Rayleigh-B\'enard convection).
This functional involves a double-well potential energy and a regularization term with higher order derivatives:
a gradient term favoring changes in the number density and a second order term controlling such changes. As periodic states are expected to nucleate in the regime of thin film domains
(which allow for elastic and plastic deformations), the order parameter is considered here as periodic in the in-plane variables together with a null-flux condition in the vertical direction.
Several works and numerical simulations illustrated the efficiency of this model to study crystallization and other phenomena, such as crystal growth \cite{EKHG}, 
homogeneous nucleation  \cite{BV}, heterogeneous nucleation, grain growth and
crack propagation for ductile material  \cite{EG}. 

An important research direction concerns the study of (global) minimizers of this energy functional according to several parameters of the system. It is expected that the minimizers are
trivially constant in some parameter regime, and they exhibit stripes or hexagonal type structures in other regimes.
Finding analytically the exact curves separating such parameter regions is a big challenge and remains still open, despite some rigorous attempts such as the paper \cite{SCN} where 
some bounds on the order-disorder phase transition were obtained by a numerical algorithm.
An attempt was also conducted in this direction for a similar type energy,
that is the Ohta-Kawasaki problem \cite{VW} using numerics that take into account the impact of domain size optimization. In \cite{BPBA}, the authors studied the existence of bifurcation branches from the trivial solution with a constraint on the Hamiltonian in the one dimensional case.
We mention the work \cite{DSSS} on the Swift-Hohenberg equation where the authors studied stability 
of the hexagonal patterns and transitions to different solutions like stable or unstable rolls.
Also, we refer to \cite{PT} where the authors studied an extended Fisher-Kolmogorov equation, finding conditions on a fourth order
term that permits existence of different type of one-dimensional periodic solutions exhibiting a countable number of kinks.
We discuss in Remark \ref{rem:sym} the question of finding the parameter region where the minimizers exhibit stripes, in particular, when they are one-dimensional symmetric.

The aim of this paper is to determine the exact parameter curve for the phase transition between the uniform state and the non-trivial states  for the PFC model, by providing a necessary and
sufficient condition for global minimality of the constant state. We also discuss the case of the Ohta-Kawasaki model.

\bigskip

\nd {\bf Model}. The $3$-dimensional domain considered here is periodic in the in-plane coordinates $y'=(y_1, y_2)$ with the period $L>0$ and of thickness $T>0$ in the vertical coordinate $y_3$;
the prototype of the cell is denoted by
$$
y=(y',y_3)\in {\cal D}=[0, L)^2 \times (0,T),
$$
where $[0,L)^2$ stands for the $2$-dimensional torus of length $L$.
For the scalar order-parameter $\Phi\in H^2({\cal D})$ that is $L$-periodic in the in-plane variables and corresponding to the local mass density, the following free-energy functional is defined
\beq
\label{functio}
{\cal F}(\Phi)=\int_{\cal D}  \bigg(\frac{1}{2}(\alpha \Phi+ \Delta \Phi)^2 + W(\Phi) \bigg)\, dy
\eeq
where $\alpha\in \R$ is a fixed constant and $W:\R\into \R$ is a continuous potential. Note that the above quantity ${\cal F}(\Phi)$ is finite for 
$\Phi\in H^2({\cal D})$ since $\Phi$ is bounded in $\Omega$ by Sobolev embedding (for more details about the well-posedness and coercivity of the functional $\cal F$,
see Lemma \ref{lem:coercive} and Remark \ref{rem:coercive} below).

In the classical PFC model \cite{BV,EG,EKHG,EPBSG,ELWGTTG}, $\alpha$ is some positive constant (usually considered in the numerical simulations equal to 1) and $W(\Phi) = \frac{1}{4}(\Phi^2-a)^2 $
is a double-well potential favoring the two phases $\pm \sqrt a$ for some constant $a>0$. 
The difficulty in treating this model in the case of positive $\alpha$ resides in the possibility of losing the coercivity of the functional $\cal F$ (as we point out in  Remark \ref{rem:coercive} below).
When $\alpha$ is nonpositive, we recover the extended Fisher-Kolmogorov type model \cite{DW} introduced for the study of some bistable physical systems,
which is a higher order generalization of the well known Allen-Cahn model for phase-transitions. Therefore, the case $\alpha\leq 0$ is somehow easier to treat due to this coercivity issue. 
\medskip

\nd {\bf Boundary conditions}. 
The number density of particles $\Phi:{\cal D} \rightarrow \R$ is supposed to be $L$-periodic in the in-plane variable $y'$, i.e., 
$$\Phi(y_1+L, y_2, y_3)= \Phi(y_1, y_2+L, y_3)= \Phi(y) \,  \textrm{ for every } \, y'\in \R^2, \, y_3\in (0,T).$$ 
On the top and bottom surfaces, a null-flux condition is imposed
$$
\partial_3 \Phi = 0 \textrm{ in } [0,L)^2 \times \{0, T\},
$$
where $\partial_3$ is the partial derivative in the vertical direction $y_3$. 
 This condition physically expresses a finite deposition rate (see \cite{EPBSG}) and will make the limit number density
 of particles to be $2$-dimensional in our thin-film regime.

\medskip

\nd {\bf Mass constraint}.
The following constant mass constraint is imposed on every order-parameter $\Phi\in H^2(\cal D)$:
$$
-\hspace{-2.4ex}\int_{\cal D} \Phi \, dy = m,
$$
where $m\in \R$ is a fixed constant.

\medskip

\nd {\bf Aim}.  We want to analyze the behavior of the energy $\cal F$ and its minimizers in the asymptotic thin-film regime where the relative thickness $\frac T L$ is very small.
First, we will employ the $\Gamma$-convergence method in order to deduce a reduced $2$-dimensional model that catches the asymptotic behavior of $\cal F$; second,
we will analyze the minimizers of the $\Gamma$-limit, more precisely, we will give a necessary and sufficient condition that guarantees that the uniform state $m$ is the
(unique) global minimizer of the limit functional.

\subsection{Dimension reduction: $\Gamma$-convergence result.}
The $\Gamma$-convergence technique is the usual way to carry out the dimension reduction and was already fruitful for energies involving higher order terms (see e.g. \cite{CF, FM, HPS, LZ1}).
We recall that a sequence of functionals $(G_n)_n$ defined on a topologic space $X$ with values into $\mathbb{R}\cup \{+\infty\}$ is $\Gamma$-converging to the limit functional
$G_0$ with respect to the topology of $X$ if and only if the following two conditions are satisfied for every $ x \in X$:
\begin{displaymath}
                \begin{cases}
                    \forall\; x_n\rightarrow x , \liminf_{n\to \infty} G_n (x_n) \geq G_0 (x) ,\\
                    \exists\; x_n\rightarrow x , G_n (x_n)\rightarrow G_0 (x) \textrm{ as } n\to \infty.
                \end{cases}
\end{displaymath} 
\nd {\bf Thin-film regime}. We consider the thin-film regime
\beq
\label{regime}
h:=\frac T L\to 0, \quad L\to 1.
\eeq
In order to carry out the asymptotic analysis,
we rescale the problem as follows: 

\medskip

\nd {\bf Scaling}. We consider the new variables
$$x_1=\frac{y_1}{L}, \quad x_2=\frac{y_2}{L},\quad x_3=\frac{y_3}{T}$$ so that $y\in {\cal D}$ if and only if 
$$x\in \Omega= [0,1)^2\times (0,1),$$ 
where the reference domain has the $2$-dimensional torus $[0,1)^2$ as basis and unit thickness in the vertical direction $x_3$.
The order parameter $\Phi:{\cal D} \rightarrow \R$ is rescaled as follows 
$$\varphi(x):=\Phi(L x_1,L x_2, T x_3), \quad x=(x_1, x_2, x_3)\in \Omega$$
for the rescaled order-parameter $\varphi:\Omega\to \R$. The nondimensionalized energy functional $\frac{1}{L^2 T} {\cal F}(\Phi)$ writes in terms of $\varphi$ as follows:
$$
F_{L, h}(\varphi):= \int_\Omega \bigg( \frac{1}{2} (\alpha \f + \frac1{L^2} \Delta' \f+ \frac{1}{L^2h^2} \partial_{33}\f)^2+ W(\f)\bigg)\, dx,
$$ 
where 
we denoted the in-plane laplacian by $\Delta' \f= \partial_{11}\f+ \partial_{22}\f$.

The boundary conditions transfer to the rescaled configuration $\f:\Omega\to \R$, i.e., $\f$ is $1$-periodic in the in-plane variable $x'=(x_1,x_2)$ and satisfies the zero
Neumann boundary condition on the top and bottom surfaces 
\beq
\label{null_flux}
\partial_3 \f = 0 \textrm{ in } [0,1)^2 \times \{0, 1\}.
\eeq
Also, the mass constraint on $\Phi$ transfers to $\f$ as 
$$\int_{\Omega} \f \, dx=m.$$ 
Therefore, we denote the set of admissible configurations by
$$V=\left\{\f \in H^2(\Omega)\, :\, \text{\eqref{null_flux} holds and }\int_{\Omega} \f \, dx=m\right\}.$$
Note that $V$ is a convex set in the space $H^2(\Omega)$. In the following, we restrict our functionals $F_{L, h}$ to the set $V$ endowed with the weak topology in $H^2(\Omega)$.

\medskip

\nd {\bf $\Gamma$-convergence}.
The aim is to prove that in the asymptotic regime \eqref{regime} the $\Gamma$-limit of functionals $(F_{L,h})$ on $V$ is given by
$$
F_*(\f)=\begin{cases}
\int_\Omega \big(\frac{1}{2} (\alpha \f + \Delta' \f)^2+ W(\f)\big)\, dx
&\quad\text{if } \f \in V_*,\\
+\infty& \quad \textrm{if } \f \in V\setminus V_*,
\end{cases}
$$
where $V_*$ is the subset of functions in $V$ that are invariant in the vertical direction, i.e.,
$$V_*=\left\{\f \in H^2(\Omega)\, :\,  \partial_3\f=0 \text{ in $\Omega$ and }\int_{\Omega} \f \, dx=m\right\}.$$
Note that for a configuration $\f\in V_*$, the functional $F_*$ corresponds to a $2$-dimensional functional on the torus $\mathbb{T}^2=[0,1)^2$:
$$F_*(\f)=\int_{\mathbb{T}^2} \bigg(\frac{1}{2} (\alpha \f + \Delta' \f)^2+ W(\f)\bigg)\, dx'.$$

\begin{Theorem}\label{TheoremGammaLimit}
Let $m\in \R$ and $(L_n)_n, (h_n)_n \subset (0, \infty)$ be two sequences such that $L_n\to 1$ and $h_n\to 0$ as $n\to \infty$. If 
$\alpha\in \R$ satisfies\footnote{If $k=(2\pi s, 2\pi t)$, then its Euclidian norm is denoted by $|k|^2=4\pi^2(s^2+t^2)$.}
\beq
\label{cond:alpha}
\alpha \notin \bigg\{|k|^2\, :\, k\in 2\pi \Z^2\setminus \{0\}\bigg\}
\eeq
and $W$ is a continuous potential on $\R$ bounded from below, i.e., 
\beq
\label{liminf} 
\liminf_{|s|\rightarrow +\infty} W(s)>-\infty,
\eeq
then the sequence of functionals $(F_{L_n, h_n})_n$ $\Gamma$-converges to $F_*$ in the weak $H^2(\Omega)$ topology. More precisely,

\medskip

\nd A. Compactness: If $(\f_n)_n$ is a sequence in $V$ such that $\limsup_{n\to \infty} F_{L_n,h_n}(\f_n)<\infty$, then up to a subsequence, 
$(\f_n)_n$ converges weakly in $H^2(\Omega)$ to a limit $\f_*\in V_*$.

\medskip

\nd B. Lower bound: If $(\f_n)_n\subset V$ converges weakly in $H^2(\Omega)$ to a limit 
$\f_*\in V$, then $\liminf_{n\to \infty} F_{L_n,h_n}(\f_n)\geq F_*(\f_*)$.

\medskip

\nd C. Upper bound: If $\f_*\in V$, then there exists a sequence $(\f_n)_n\subset V$ such that $\f_n\to \f_*$ strongly in $H^2(\Omega)$  and $\lim_{n\to \infty} F_{L_n,h_n}(\f_n)= F_*(\f_*)$.
\end{Theorem}

The main ingredient in the proof of the $\Gamma$-convergence is given by the coercivity of the functional $F_{L,h}$ in $V$:

\begin{Lemma}\label{lem:coercive}
Let $m\in \R$  and $W$ be a continuous potential on $\R$ with \eqref{liminf}. Then for every $\alpha \in \R$ with \eqref{cond:alpha}, there exist
$C, \eps, h_0>0$ (all depending on $\alpha$) such that for every $L\in (1-\eps, 1+\eps)$ and every $h\in (0, h_0)$ we have
$$
F_{L,h}(\f)\geq C\bigg(\|\f-m\|^2_{H^2(\Omega)}+\frac{1}{h^4} \int_\Omega (\partial_{3}\f)^2 \, dx\bigg)+ \inf W, \quad \textrm{for all } \f \in V.
$$
\end{Lemma}

Thanks to the $\Gamma$-convergence result in Theorem \ref{TheoremGammaLimit}, the minimizers of $F_{L,h}$ converge to the minimizers of the limit functional $F_*$ over $V_*$  in the regime \eqref{regime}. 
This justifies the importance of the analysis of the minimizers of the limit problem that is done in the next section.

\subsection{Optimality of the uniform state in the PFC model.}

The aim of this section is to analyze the minimizers of the $\Gamma$-limit $F_*$ over the set $V_*$. 
In \cite[Theorem 3.1]{SCN}, the authors provide a lower bound for the order-disorder phase transition which is illustrated by the fact of whether or not the constant state is a global minimizer.
They also provide numerical results for this phase transition. 
In our analysis, we provide the exact phase transition in the case of the double-well potential $W$ obtaining a necessary and sufficient condition for global minimality of the uniform state.

Note that if $W$ is a $C^1$ potential, then the Euler-Lagrange equation satisfied by a critical point $\f_*$ of $F_*$ over $V_*$ is the following:
\beq
\label{E-L}
(\Delta')^2\f_*+2\alpha \Delta' \f_*+\alpha^2\f_*+\frac{d W}{d \f}(\f_*)=\alpha^2 m+\int_{\Omega} \frac{d W}{d \f}(\f_*)\, dx,
\eeq
where the right-hand side is due to the constant mass constraint. The necessary and sufficient condition for the uniform state $\f_*=m$ to be a stable critical point of $F_*$ over $V_*$ for $C^2$ potentials $W$ is:
\beq
\label{cond_min}
\frac{d^2 W}{d \f^2}(m)+\min_{k\in 2\pi \Z^2, k\neq 0} (\alpha-|k|^2)^2\geq 0.
\eeq
However, in order to ensure that $\f_*=m$ is a global minimizer of $F_*$ over $V_*$ we need a stronger assumption that is related to the following optimal constant in the $2$-dimensional torus $\mathbb{T}^2$:
\begin{align*}
P_{N=2}&:=\inf\left\{ \int_{\mathbb{T}^2} \bigg((\alpha u+\Delta u)^2
+ \frac{d^2 W}{d \f^2}(m) u^2 \bigg) \, dx  \int_{\mathbb{T}^2} u^4\, dx \, : \, \right.\\
\nonumber
&\hspace{5cm}\, \left. u:\mathbb{T}^2\to \R, \, \int_{\mathbb{T}^2} u^3\, dx=1, \, \int_{\mathbb{T}^2} u\, dx=0  \right\}.
\end{align*}
(In Proposition \ref{prop} below, we will relate the above constant $P_2$ with the condition \eqref{cond_min}.)
Our main result provides a necessary and sufficient condition for the state $\f_*=m$ to be a
(unique) global minimizer of $F_*$ over $V_*$ in the case of the double-well potential $W$ which is an improvement of the result \cite[Theorem 3.1]{ChoksiPe} and \cite[Proposition 3.1]{Glasner}. 

\begin{Theorem}\label{the:optimal}
Let $m, \alpha \in \R$ and $W\in C^2(\R)$. 

1. The uniform state $\f_*=m$ is a stable critical point of $F_*$ over $V_*$ if and only if 
\eqref{cond_min} holds true. 

\medskip

2.  Assume that $W\in C^4(\R)$ satisfies $\frac{d^4 W}{d \f^4}\geq w^2$ in $\R$ for some constant $w>0$. 
If \eqref{cond_min} holds true and
\beq
\label{ineg22}
P_2\geq \frac1{3w^2}\left(\frac{d^3 W}{d \f^3}(m) \right)^2,
\eeq
then $m$ is a global minimizer of $F_*$ over $V_*$. Moreover, if the inequality \eqref{ineg22} is strict, then $m$ is the unique global minimizer of  $F_*$ over $V_*$.

\medskip

3. Let $W\in C^4(\R)$ such that $\frac{d^4 W}{d \f^4}=w^2$ in $\R$ for some constant $w>0$ and assume that the inequality in \eqref{cond_min} is strict. Then $m$ is {\bf not} a global minimizer of $F_*$ over $V_*$ provided that
$P_2< \frac1{3w^2}\left(\frac{d^3 W}{d \f^3}(m) \right)^2.
$
\end{Theorem}

In Section \ref{sec:optima} we prove the above result in any dimension $N\geq 1$. Moreover, in Remark~\ref{potPFC} below, we interpret the necessary and sufficient condition in Theorem \ref{the:optimal}
 in the case of the potential $W(\f)=\frac14(\f^2-a)^2$ with $a>0$ typical for the PFC model and compare it with the works \cite{ChoksiPe, EKHG, Glasner, SCN}.

\subsection{Uniform state for the  Ohta-Kawasaki energy.}

We present now Theorem \ref{the:optimal} in the case of the Ohta-Kawasaki functional
for self-assembly of diblock copolymers (see e.g. \cite{OK, ChoksiPe}).
Let $\mathbb{T}^N=[0,1)^N$ be the $N$-dimensional torus and consider the set 
of periodic configurations $\phi$ of average $m\in \R$: 
$$H_m^1(\mathbb{T}^N)=\bigg\{\phi\in H^1(\mathbb{T}^N)\, :\, \int_{\mathbb{T}^N} \phi \, dx=m\bigg\}.$$
The Ohta-Kawasaki energy is defined as
$$
{\cal E}(\phi)= \int_{\mathbb{T}^N} \bigg(\frac{1}{2\gamma^2}|\nabla \phi|^2
+ \frac12 |\nabla(-\Delta)^{-1}(\phi-m)|^2+W(\phi)\bigg) \, dx, \quad \phi\in H_m^1(\mathbb{T}^N),
$$
where $W$ is a $C^2$ potential and $\gamma>0$ is a constant parameter. The second term in the above functional can be rewritten as
$$\int_{\mathbb{T}^N} |\nabla(-\Delta)^{-1}(\phi-m)|^2\, dx=\|\phi-m\|^2_{\dot{H}^{-1}(\T)}=\int_{\mathbb{T}^N} |\nabla \psi|^2\, dx,$$
where $\psi$ is the unique solution of the problem
\beq
\label{psi}
-\Delta \psi=\phi-m \textrm{ in }  \mathbb{T}^N \quad \textrm{ and } \quad \int_{\mathbb{T}^N} \psi=0.
\eeq
The necessary and sufficient condition for the uniform state $\phi_*=m$ to be a stable critical point of $\cal E$ over $H_m^1(\mathbb{T}^N)$ is
\beq
\label{cond_OK}
\frac{d^2 W}{d \phi^2}(m)+\min_{k\in 2\pi \Z^N, k\neq 0} \left(\frac{|k|^2}{\gamma^2}+\frac1{|k|^2}\right)\geq 0.
\eeq
Furthermore, the necessary and sufficient condition for the uniform state $\phi_*=m$ to be a global minimizer of $\cal E$ over $H_m^1(\mathbb{T}^N)$ is related with the following optimal constant: for the fixed constants $\gamma^2>0$ and $\frac{d^2 W}{d \phi^2}(m)\in \R$, let 
\begin{align}
\label{qn}
Q_N&:=\inf\left\{ \int_{\mathbb{T}^N} \bigg(\frac{1}{\gamma^2}|\nabla u|^2
+  |\nabla(-\Delta)^{-1}u|^2
+ \frac{d^2 W}{d \phi^2}(m) u^2 \bigg) \, dx  \int_{\mathbb{T}^N} u^4\, dx \, : \right.\\
\nonumber
&\hspace{5cm}\, \left. u:\mathbb{T}^N\to \R, \, \int_{\mathbb{T}^N} u^3\, dx=1, \, \int_{\mathbb{T}^N} u\, dx=0  \right\}.
\end{align}
In the next Proposition, we prove how the optimal constant $Q_N$ in \eqref{qn} is related with \eqref{cond_OK}. We also provide a sufficient condition in order that $Q_N$ is achieved in \eqref{qn}.

\begin{Proposition}\label{prop2} 
Let $\gamma^2>0$ and $\frac{d^2 W}{d \phi^2}(m)\in \R$ be fixed constants.

\nd 1. If \eqref{cond_OK} holds true, then 
\beq
\label{ineg_qn}
Q_N\geq \frac{d^2 W}{d \phi^2}(m)+\min_{k\in 2\pi \Z^N, k\neq 0} \left(\frac{|k|^2}{\gamma^2}+\frac1{|k|^2}\right).
\eeq

\nd 2. If the inequality in \eqref{cond_OK} is strict, 
then the infimum in \eqref{qn} is achieved provided that $N<6$. Moreover, the inequality in \eqref{ineg_qn} is strict, i.e., 
$$
Q_N> \frac{d^2 W}{d \phi^2}(m)+\min_{k\in 2\pi \Z^N, k\neq 0} \left(\frac{|k|^2}{\gamma^2}+\frac1{|k|^2}\right).
$$
\end{Proposition}

We prove the following necessary and sufficient condition for the uniform state $\phi_*=m$ to be the (unique) global minimizer of $\cal E$ over $H_m^1(\mathbb{T}^N)$ in the case of the double-well potential $W$.

\begin{Theorem}\label{the:optimal_OK}
Let $m \in \R$, $\gamma>0$ and $W\in C^2(\R)$.

1. The uniform state $\phi_*=m$ is a stable critical point of $\cal E$ over $H_m^1(\mathbb{T}^N)$ if and only if \eqref{cond_OK} holds true.

\medskip

2.  Assume that $W\in C^4(\R)$ satisfies $\frac{d^4 W}{d \phi^4}\geq w^2$ in $\R$ for some constant $w>0$. Then $m$ is a global minimizer of $\cal E$ over $H_m^1(\mathbb{T}^N)$ if \eqref{cond_OK} holds true and
\beq
\label{cond_optima_OK}
Q_N\geq \frac1{3w^2}\left(\frac{d^3 W}{d \phi^3}(m) \right)^2.
\eeq
Moreover, if the inequality in \eqref{cond_optima_OK} is strict, then $m$ is the unique global minimizer of $\cal E$ over $H_m^1(\mathbb{T}^N)$.

\medskip

3. Assume that $W\in C^4(\R)$ has constant $4$-derivative, i.e., $\frac{d^4 W}{d \f^4}=w^2$ in $\R$ for some constant $w>0$. If $N<6$ and the inequality in \eqref{cond_OK} is strict, then $m$ is {\bf not} a global minimizer of $\cal E$ over $H_m^1(\mathbb{T}^N)$ provided that
\eqref{cond_optima_OK} fails to be true.
\end{Theorem}

\section{$\Gamma$-convergence result. Proof of Theorem \ref{TheoremGammaLimit} and Lemma~\ref{lem:coercive}. }

\begin{proof}[Proof of Lemma \ref{lem:coercive}]
We develop the first integrant in $F_{L,h}$:
\begin{align*}
(L^2\alpha \f+\Delta' \f +\frac1{h^2}\partial_{33}\f)^2&=(L^2\alpha \f+\Delta' \f)^2 +\frac1{h^4}(\partial_{33}\f)^2\\
&\quad \quad +\frac{2L^2\alpha}{h^2}\f \partial_{33}\f+\frac2{h^2} \Delta' \f  \partial_{33}\f, \quad \textrm{for } \f \in V.
\end{align*}
\nd {\it Step 1. Integrating the crossing terms}. 
Using \eqref{null_flux}, integration by parts leads
$$\int_\Omega \f \partial_{33}\f \, dx =-\int_\Omega (\partial_{3}\f)^2 \, dx.$$
Using the periodicity of $\f$ in $x'=(x_1, x_2)$ and \eqref{null_flux}, integration by parts applied first in $x_3$ direction and then in $x_j$-direction for $j=1,2$ yields
\footnote{This computation is carried out for $\f$ smooth in $V$ and the result follows for general $\f\in V$ by a standard density argument.}
\beq
\label{crossing}
\int_\Omega \partial_{jj}\f \partial_{33}\f \, dx = \int_\Omega (\partial_{j3}\f)^2 dx, \quad j=1,2,
\eeq
so that
$$\int_\Omega \Delta' \f  \partial_{33}\f\, dx\geq 0.$$

\nd {\it Step 2. We prove that}
$$\int_\Omega (L^2\alpha \f+\Delta' \f)^2 \, dx\geq  L^4\alpha^2 m^2+\inf_{k\in 2\pi \Z^2, k\neq 0} \big(\frac{L^2\alpha}{|k|^2}-1\big)^2 \int_\Omega (\Delta' \f)^2\, dx.$$
Indeed, as $\f\in H^2(\Omega)$ is $1$-periodic in $x'$-variables, the Fourier series expansion of $\f$ writes
\beq
\label{Fourier}
\f(x)= a_0(x_3) + \sum_{k\in 2\pi \Z^2, k\neq 0} \big(a_{k}(x_3)\cos(k\cdot x')+b_{k}(x_3)\sin(k\cdot x')\big), \quad x=(x',x_3)\in \Omega,
\eeq
where $a_0, a_k, b_k\in H^2((0,1))$ for every $k\in 2\pi \Z^2, k\neq 0$ and $\cdot$ denotes the scalar product in $\R^2$. One computes that  
\begin{align*}
&L^2\alpha \f(x)+\Delta' \f(x)=L^2\alpha a_0(x_3)\\
&\quad \quad +\sum_{k\in 2\pi \Z^2, k\neq 0}(L^2\alpha-|k|^2)
\big(a_{k}(x_3)\cos(k\cdot x')+b_{k}(x_3)\sin(k\cdot x')\big), \quad \textrm{for a.e. }  x\in \Omega.
\end{align*}
Then Plancherel's identity and Jensen's inequality yield
\begin{align*}
\int_\Omega (L^2 \alpha \f+\Delta' \f)^2\, dx&=L^4\alpha^2 \int_0^1 a_0^2\, dx_3\\
&\quad + 
\frac{1}{2}\sum_{k\in 2\pi \Z^2, k\neq 0}(L^2\alpha-|k|^2)^2 \int_0^1 (a_{k}^2+b_{k}^2)\, dx_3\\
&\geq  L^4\alpha^2 \big(\int_0^1 a_0\, dx_3\big)^2+\inf_{k\in 2\pi \Z^2, k\neq 0} \big(\frac{L^2\alpha}{|k|^2}-1\big)^2 \int_\Omega (\Delta' \f)^2\, dx,
\end{align*}
which proves the desired inequality since $\int_0^1 a_0\, dx_3=\int_\Omega\f\, dx=m$. Note that if $\alpha<0$, the above infimum equals $1$ (and is not achieved by
any $k\in 2\pi \Z^2$); if $\alpha\geq 0$, the above infimum is achieved for some $k_L\in 2\pi \Z^2\setminus \{0\}$.

\medskip

\nd {\it Step 3. Conclusion.}
We recall the Poincar\'e inequality for $H^1_0$ functions on the interval $(0,1)$:
\beq
\label{poinc}
\int_0^1 \big(\frac{d}{ds}u\big)^2\, ds\geq \pi^2 \int_0^1 u^2\, ds, \quad \textrm{for every } u\in H^1_0((0,1)).
\eeq
Applying it for $\partial_3 \f(x_1,x_2,.)\in H^1_0((0,1))$ for a.e. $(x_1,x_2)\in [0,1)^2$, we deduce that
$$\int_\Omega \bigg(\frac1{h^4}(\partial_{33}\f)^2-\frac{2L^2\alpha}{h^2} (\partial_{3}\f)^2 \bigg) \, dx
\geq \frac1{2h^4}\int_\Omega (\partial_{33}\f)^2 \, dx+\frac{\pi^2-4h^2L^2\alpha}{2h^4} \int_\Omega (\partial_{3}\f)^2 \, dx.$$
Combined with Steps 1 and 2, we conclude that there exist $C, \eps, h_0>0$ (all depending on $\alpha$) such that for every $L\in (1-\eps, 1+\eps)$ and every $h\in (0, h_0)$
we have $\inf_{k\in 2\pi \Z^2, k\neq 0} \big(\frac{L^2\alpha}{|k|^2}-1\big)^2\geq C>0$ (thanks to \eqref{cond:alpha}) and 
\beq
\label{ineg_c}
F_{L,h}(\f)\geq C\int_\Omega \bigg((\Delta'\f)^2+(\partial_{33}\f)^2+\frac{1}{h^4}(\partial_{3}\f)^2 \bigg)\, dx+ \inf W.
\eeq
Moreover, by \eqref{Fourier}, we deduce that $\|\Delta'\f\|_{L^2(\Omega)}\geq \|\nabla'\f\|_{L^2(\Omega)}$ with $\nabla'=(\partial_1, \partial_2)$ and by \eqref{poinc},
$\|\partial_{33}\f\|_{L^2(\Omega)}\geq \pi 
\|\partial_3\f\|_{L^2(\Omega)}$; the conclusion then follows by \eqref{crossing} and the Poincar\'e-Wirtinger inequality  
$\|\nabla\f\|_{L^2(\Omega)}\geq C  \|\f-m\|_{L^2(\Omega)}$.
\end{proof}

\bigskip

\begin{Remark}
\label{rem:coercive}
1. Note that $H^2({\cal D})$ is not in general the whole space of finite energy configurations of the functional $\cal F$ in \eqref{functio}. Indeed, if $\alpha=|k_0|^2$ for some $k_0\in 
\frac{2\pi}{L}\Z^2\setminus \{0\}$, then by setting $\Phi_\lambda(y)=m+\lambda \sin(k_0\cdot y')$ for $y=(y', y_3)\in {\cal D}$, we have that $\alpha \Phi_\lambda+\Delta \Phi_\lambda=\alpha m$.
If $W\equiv 0$, it implies that 
${\cal F}(\Phi_\lambda)$ fails to bound $\|\Delta \Phi_\lambda\|^2_{L^2(\cal D)}$ as $\lambda\to \infty$.

\medskip

2. The coercivity result in Lemma \ref{lem:coercive} holds for more general continuous potentials $W$ for which there exist two positive constants $C_\alpha, c>0$ such that
$$W(t)\geq -C_\alpha t^2-c, \quad \textrm{ for all } t \in \R$$
(in particular, \eqref{liminf} could fail).  
The constant $C_\alpha>0$ (depending on $\alpha$) needs to satisfy the following bound
$$C_\alpha<2\pi^2 \inf_{k\in 2\pi \Z^2, k\neq 0} \big(\frac{\alpha}{|k|^2}-1\big)^2$$
and $\alpha \in \R$ is such that \eqref{cond:alpha} holds true. 
Indeed, this follows by the proof of Lemma \ref{lem:coercive} combined with 

$\bullet$ the Poincar\'e inequality
$$\int_{\Omega} (\Delta' \f)^2\, dx\geq 4\pi^2  \int_{\Omega} (\f-a_0)^2\, dx=4\pi^2 \int_{\Omega} \f^2\, dx-4\pi^2 \int_0^1 a_0(x_3)^2\, dx_3, \quad \textrm{for all }\f \in V,$$
which follows by the Fourier expansion \eqref{Fourier} with $a_0(x_3)=\int_{[0,1)^2} \f(x', x_3)\, dx'$ for every $x_3\in (0,1)$;

$\bullet$ and the following inequalities
$$\int_\Omega (\partial_3 \f)^2\, dx\geq \int_0^1 (\partial_3 a_0)^2\, dx_3 \geq 4\pi^2 \int_0^1 (a_0-m)^2\, dx_3=4\pi^2 \bigg(\int_0^1 a_0^2\, dx_3-m^2\bigg),$$
(where we used the Jensen and Poincar\'e-Wirtinger inequalities for $\int_0^1 a_0\, dx_3=m$).
 \end{Remark}

\bigskip

Now we prove the $\Gamma$-convergence result in Theorem \ref{TheoremGammaLimit}.

\begin{proof}[Proof of Theorem \ref{TheoremGammaLimit}]
We divide the proof in several steps:

\medskip

\nd {\it Step 1. Proof of point A. (Compactness).} By Lemma \ref{lem:coercive}, we know that $(\f_n)$ is bounded in $H^2(\Omega)$; therefore, up to a subsequence,
$\f_n\rightharpoonup \f_*$ weakly in $H^2(\Omega)$. As $H^2(\Omega)$ is compactly embedded in $H^1(\Omega)$ and in $L^1(\Omega)$, we deduce that the mass constraint passes to the limit 
(i.e., $\int_\Omega \f_*\, dx=m$) as well as $\partial_3 \f_n\to \partial_3 \f_*$ in $L^2(\Omega)$ (up to a subsequence). Moreover, by Lemma \ref{lem:coercive}, we know that 
$\| \partial_3 \f_n\|_{L^2(\Omega)}\to 0$; therefore, we deduce that $\partial_3 \f_*=0$ in $\Omega$. We conclude that $\f_*\in V_*$.   

\medskip

\nd {\it Step 2. Proof of point B. (Lower bound).} 
Since $\f_n\rightharpoonup \f_*$ weakly in $H^2(\Omega)$ which is compactly embedded in $L^\infty(\Omega)$, we know that up to a subsequence, $\f_n\to \f_*$ uniformly in $\Omega$.
In particular, $W(\f_n)\to W(\f_*)$ uniformly in $\Omega$ (because $W$ is continuous). As in the proof of Lemma \ref{lem:coercive} (see Step 1), we write
\begin{align*}
\int_\Omega (L_n^2\alpha \f_n&+\Delta' \f_n +\frac1{h_n^2}\partial_{33}\f_n)^2\, dx=
\int_\Omega (L_n^2\alpha \f_n+\Delta' \f_n)^2 \, dx\\
&+\frac1{h_n^4}\int_\Omega \bigg((\partial_{33}\f_n)^2-2h_n^2 L_n^2\alpha (\partial_{3}\f_n)^2\bigg)\, dx+
\frac2{h_n^2}\int_\Omega |\nabla' \partial_3 \f_n|^2 \, dx,
\end{align*}
where $\nabla'=(\partial_1, \partial_2)$. 
Using the Poincar\'e inequality \eqref{poinc}, we know that for $n$ large the last two integrals are nonnegative (because $2h_n^2 L_n^2\alpha\to 0$,
so it is less than the constant $\pi^2$ in \eqref{poinc} as $n\to \infty$). Since
$L_n^2\alpha \f_n+\Delta' \f_n \rightharpoonup \alpha \f_*+\Delta' \f_*$ weakly in $L^2(\Omega)$, the lower semicontinuity of
$\|\cdot\|_{L^2(\Omega)}^2$ yields the conclusion.  

\medskip

\nd {\it Step 3. Proof of point C. (Upper bound).} 
Let $\f_*\in V$. We set $\f_n:=\f_*$. If $\f_*\notin V_*$ (i.e., $\|\partial_3 \f_*\|_{L^2(\Omega)}\neq 0$), then 
$F_{L_n, h_n}(\f_*)\to \infty$ (by Lemma \ref{lem:coercive}). Otherwise, $\f_*\in V_*$ and 
$$F_{L_n,h_n}(\f_n)=\int_\Omega \bigg( \frac{1}{2L_n^4} (L_n^2\alpha \f_* + \Delta' \f_*+ \frac{1}{h_n^2} \underbrace{\partial_{33}\f_*}_{=0})^2+ W(\f_*)\bigg)\, dx
\rightarrow F_*(\f_*)$$
by dominated convergence theorem (as $\f_*\in V_*\subset H^2(\Omega)$).
\end{proof}

\section{Optimality of the uniform state. Proof of Theorems \ref{the:optimal} and \ref{the:optimal_OK}.}
\label{sec:optima}

\subsection{The case of the PFC model.}

In this section, we give necessary and sufficient conditions on the parameter $\alpha$ and on the potential $W$ that guarantee the global minimality of the constant state $m$
for the $\Gamma$-limit $F_*$ over the set $V_*$. In fact, we will work in the general context of  the $N$-dimensional torus  $$\mathbb{T}^N=[0,1)^N$$ with $N\geq 1$ and the set 
of periodic configurations $\phi$ of average $m\in \R$ : 
$$H_m^2(\mathbb{T}^N)=\{\phi\in H^2(\mathbb{T}^N)\, :\, \int_{\mathbb{T}^N} \phi \, dx=m\}.$$
The corresponding functional is
$$
{\cal F}(\phi)= \int_{\mathbb{T}^N} \bigg(\frac{1}{2}(\alpha\phi+\Delta \phi)^2
+ W(\phi)\bigg) \, dx, \quad \phi\in H_m^2(\mathbb{T}^N),
$$
where $W$ is a $C^2$ potential, $\alpha$ is a constant parameter and $\Delta$ is the Laplacian operator in $\R^N$.
For the fixed constants $\alpha\in \R$ and $\frac{d^2 W}{d \phi^2}(m)\in \R$, we denote 
\begin{align}
\label{cn}
P_N&:=\inf\left\{ \int_{\mathbb{T}^N} \bigg((\alpha u+\Delta u)^2
+ \frac{d^2 W}{d \phi^2}(m) u^2 \bigg) \, dx  \int_{\mathbb{T}^N} u^4\, dx \, : \right.\\
\nonumber
&\hspace{5cm}\, \left. u:\mathbb{T}^N\to \R, \, \int_{\mathbb{T}^N} u^3\, dx=1, \, \int_{\mathbb{T}^N} u\, dx=0  \right\}.
\end{align}

We start by proving the following result that relates the optimal constant $P_N$ in \eqref{cn} with the condition of stability of the uniform state $\phi_*=m$ (that is \eqref{cond_min_N} below).
We also give a sufficient condition in order that the infimum in $P_N$ is achieved in \eqref{cn}.

\begin{Proposition}\label{prop} 
Let $\alpha\in \R$ and $\frac{d^2 W}{d \phi^2}(m)\in \R$ be fixed.

\nd 1. If 
\beq
\label{cond_min_N}
\frac{d^2 W}{d \phi^2}(m)+\min_{k\in 2\pi \Z^N, k\neq 0} (\alpha-|k|^2)^2\geq 0,
\eeq
then 
\beq
\label{ineg_pn}
P_N\geq \frac{d^2 W}{d \phi^2}(m)+\min_{k\in 2\pi \Z^N, k\neq 0} (\alpha-|k|^2)^2.
\eeq

\nd 2. If the inequality in \eqref{cond_min_N} is strict, then the infimum in \eqref{cn} is achieved provided that $N<12$. Moreover, the inequality in \eqref{ineg_pn} is strict, i.e., 
$$
P_N>\frac{d^2 W}{d \phi^2}(m)+\min_{k\in 2\pi \Z^N, k\neq 0} (\alpha-|k|^2)^2.
$$

\end{Proposition}

\begin{proof}[Proof of Proposition \ref{prop}] Assume that \eqref{cond_min_N} holds true. We divide the proof in several steps:
\medskip

\nd {\it Step 1. Proof of \eqref{ineg_pn}}. For  $u \in H^2({\mathbb{T}^N})$ of zero average on $\T$, we write the following Fourier series expansion  
$$
u(x)=\sum_{k\in 2\pi \Z^N, k\neq 0} \big(a_{k}\cos(k\cdot x))+b_{k}\sin(k\cdot x)\big), \quad x\in \mathbb{T}^N,
$$
where $a_k, b_k\in \R$ for  $k\in 2\pi \Z^N\setminus\{0\}$ and $\cdot$ is the scalar product in $\R^N$.
By Plancherel's identity, we have
\begin{align}
\nonumber
\int_{\mathbb{T}^N} (\alpha u+\Delta u)^2 + \frac{d^2 W}{d \phi^2}(m) u^2 \, dx&= \frac{1}{2}\sum_{k\in 2\pi \Z^N\setminus\{0\}} \bigg((\alpha -|k|^2)^2+\frac{d^2 W}{d \phi^2}(m) \bigg)(a_{k}^2+b_{k}^2)\\
\label{interm1}
&\geq \bigg(\frac{d^2 W}{d \phi^2}(m)+\min_{k\in 2\pi \Z^N, k\neq 0} (\alpha-|k|^2)^2\bigg) \int_{\T} u^2\, dx,
\end{align}
which is a nonnegative quantity thanks to \eqref{cond_min_N}. By the H\"older inequality
\beq
\label{interm2}
\int_{\T} u^4\, dx\int_{\T} u^2\, dx\geq \bigg(\int_{\T} |u|^3\, dx\bigg)^{2}\geq \bigg(\int_{\T} u^3\, dx\bigg)^{2}.
\eeq
Therefore, one deduces the conclusion in point 1.

\medskip

For the rest of the proof, we assume that the inequality in \eqref{cond_min_N} is strict.
\medskip

\nd {\it Step 2. Every minimizing sequence in \eqref{cn} is bounded in $H^2(\mathbb{T}^N)$}. 
{Indeed,} let $(u_n)_n$ be a minimizing sequence in \eqref{cn} with 
$$\int_{\mathbb{T}^N} u_n^3\, dx=1, \, \int_{\mathbb{T}^N} u_n\, dx=0,$$ 
i.e., 
$$\int_{\mathbb{T}^N} \bigg((\alpha u_n+\Delta u_n)^2
+ \frac{d^2 W}{d \phi^2}(m) u_n^2 \bigg) \, dx  \int_{\mathbb{T}^N} u_n^4\, dx \to P_N, \quad \textrm{as } n\to \infty.$$
In particular, the above left-hand side is uniformly bounded (from above). Moreover, by H\"older's inequality, we have that 
$$\int_{\T} u_n^4\, dx\geq \bigg(\int_{\T} |u_n|^3\, dx\bigg)^{4/3}\geq 1.$$ 
As by \eqref{interm1} and the strict inequality in \eqref{cond_min_N} we already know that
$$\int_{\mathbb{T}^N} \bigg((\alpha u_n+\Delta u_n)^2
+ \frac{d^2 W}{d \phi^2}(m) u_n^2 \bigg) \, dx$$ is positive, we conclude that the above quantity is uniformly bounded from above in $n$. Combined again with \eqref{interm1} and the strict inequality 
in \eqref{cond_min_N},  we deduce that $(u_n)_n$ is bounded in $L^2(\T)$. Therefore, $(\alpha u_n+\Delta u_n)_n$ is bounded in $L^2(\T)$, yielding $(\Delta u_n)_n$ is bounded in $L^2(\T)$ and we conclude that 
$(u_n)$ is bounded in $H^2(\T)$ since
$$
\|\Delta u_n\|_{L^2(\T)}\geq C\|u_n\|_{H^2(\T)},
$$
for a universal constant $C>0$, for every zero-average periodic function $u_n$.
\medskip

\nd {\it Step 3. Existence of a minimizer in \eqref{cn}}. As $(u_n)$ is bounded in $H^2(\T)$, we know that up to a subsequence, $(u_n)_n$ converges to a function $u\in H^2(\T)$ weakly in $H^2$, a.e. in $\T$ and strongly in $L^p$ for $p\in [1,3]$ (by the Sobolev compact embedding $H^2(\T)\subset L^3(\T)$ provided that $N<12$). We conclude that $u$ has zero average, $\|u\|_{L^3}=1$, $\liminf_{n\to \infty} \int_{\T} u_n^4\, dx\geq \int_{\T} u^4\, dx$ (by Fatou's lemma) and
$$\liminf_{n\to \infty} \int_{\mathbb{T}^N} \bigg((\alpha u_n+\Delta u_n)^2
+ \frac{d^2 W}{d \phi^2}(m) u_n^2 \bigg) \, dx \geq \int_{\mathbb{T}^N} \bigg((\alpha u+\Delta u)^2
+ \frac{d^2 W}{d \phi^2}(m) u^2 \bigg) \, dx
$$
as $u_n\to u$ in $L^2$, $\alpha u_n+\Delta u_n$ converges weakly in $L^2$ to $\alpha u+\Delta u$ and the $L^2$-norm is weakly lower semicontinuous. Thus, $u$ is a minimizer in \eqref{cn}. 

\medskip

\nd {\it Step 4. Proof of the strict inequality in \eqref{ineg_pn}}. Assume by contradiction that the equality holds in 
\eqref{ineg_pn}. By Step 3, \eqref{cn} has a nonvanishing minimizer $u$ (as $\|u\|_{L^3}=1$), so that the above assumption would imply 
$$\int_{\mathbb{T}^N} \bigg((\alpha u+\Delta u)^2
+ \frac{d^2 W}{d \phi^2}(m) u^2 \bigg) \, dx  \int_{\mathbb{T}^N} u^4\, dx\\
=\frac{d^2 W}{d \phi^2}(m)+\min_{k\in 2\pi \Z^N, k\neq 0} (\alpha-|k|^2)^2.$$
By Step 1, all the inequalities in \eqref{interm1} and \eqref{interm2} become equalities. In particular,  
$$\int_{\T} |u|^3\, dx= \int_{\T} u^3\, dx=1,$$
i.e., $u\geq 0$ a.e. in $\T$. As $u$ has vanishing average, it means that $u=0$ a.e. in $\T$ which contradicts the hypothesis
$\|u\|_{L^3}=1$. 
\end{proof}

Remark that \footnote{One inequality comes from \eqref{interm2}. To prove that $1$ is indeed the infimum in \eqref{eqal11}, it is enough to consider the case of dimension $N=1$: for every $n\geq 1$, let 
$v_n=n$ in $(0,\frac1n)$ and $v_n=-\frac{n}{n-1}$ in $(\frac1n,1)$. Then the sequence $$u_n=(\int_{\mathbb{T}}v_n^3)^{-\frac13}v_n$$ is a minimizing sequence in \eqref{eqal11} yielding the value $1$ for the infimum.}
\beq
\label{eqal11}
\inf\left\{ \int_{\mathbb{T}^N} u^2 \, dx  \int_{\mathbb{T}^N} u^4\, dx \, : \, 
\int_{\mathbb{T}^N} u^3\, dx=1, \, \int_{\mathbb{T}^N} u\, dx=0  \right\}=1.
\eeq
Therefore, without the hypothesis at point 2. (implying in particular, that $P_N$ is achieved for $N<12$),
it is not clear how to conclude that the inequality \eqref{ineg_pn} is strict in general. Moreover, it may happen that if the equality holds in \eqref{cond_min_N}, then $P_N=0$. Indeed, already in dimension $N=1$, if we set $\alpha=10\pi^2$, $\frac{d^2 W}{d \phi^2}(m)=-(\alpha-4\pi^2)^2$ (so, the equality holds in \eqref{cond_min_N}) and $v(x)=\cos(2\pi x)+\cos(4\pi x)$, by normalizing $v$ as
$$u=(\int_{\mathbb{T}}v^3)^{-\frac13}v,$$
we obtain that $u$ and $(\alpha u+\frac{d^2}{dx^2}u)^2+\frac{d^2 W}{d \phi^2}(m) u^2$ have zero average and 
$u^3$ has average $1$; this yields that $P_N=0$.

\medskip

We will prove now the main result which is a generalization of Theorem \ref{the:optimal}:

\begin{Theorem}\label{gen:optimal}
Let $m, \alpha \in \R$ and $W\in C^2(\R)$. 

1. The uniform state $\phi_*=m$ is a stable critical point of $\cal F$ over $H_m^2(\mathbb{T}^N)$ if and only if \eqref{cond_min_N} holds true.

\medskip

2.  Assume that $W\in C^4(\R)$ satisfies $\frac{d^4 W}{d \phi^4}\geq w^2$ in $\R$ for some constant $w>0$. Then $m$ is a global minimizer of $\cal F$ over $H_m^2(\mathbb{T}^N)$ if \eqref{cond_min_N} holds true and
\beq
\label{cond_optima}
P_N\geq \frac1{3w^2}\left(\frac{d^3 W}{d \phi^3}(m) \right)^2.
\eeq
Moreover, if the inequality in \eqref{cond_optima} is strict, then $m$ is the unique global minimizer of $\cal F$ over $H_m^2(\mathbb{T}^N)$.

\medskip

3. Assume that $W\in C^4(\R)$ satisfies $\frac{d^4 W}{d \phi^4}=w^2$ in $\R$ for some constant $w>0$. 
If $N<12$ and the inequality in \eqref{cond_min_N} is strict, then $m$ is {\bf not} a global minimizer of $\cal F$ over $H_m^2(\mathbb{T}^N)$ provided that
\eqref{cond_optima} fails to be true.
\end{Theorem}

\begin{proof}[Proof of Theorem \ref{gen:optimal}] We divide the proof in several steps:

 \medskip

\nd {\it Step 1. A Fourier expansion}. For  $\phi \in H_m^2({\mathbb{T}^N})$, we write the following Fourier series expansion  
$$
\phi(x)=m + \sum_{k\in 2\pi \Z^N, k\neq 0} \big(a_{k}\cos(k\cdot x))+b_{k}\sin(k\cdot x)\big), \quad x\in \mathbb{T}^N,
$$
where $a_k, b_k\in \R$ for  $k\in 2\pi \Z^N\setminus\{0\}$.
By Plancherel's identity, we have
\beq
\label{star}
\int_{\mathbb{T}^N} (\alpha \phi+\Delta \phi)^2 \,dx= \alpha^2m^2 + \frac{1}{2}\sum_{k\in 2\pi \Z^N\setminus\{0\}} (\alpha -|k|^2)^2 (a_{k}^2+b_{k}^2).
\eeq

 \medskip

\nd {\it Step 2. Proof of 1.} 
First, note that $\phi_*=m$ is indeed a critical point of $\cal F$ over $H_m^2(\mathbb{T}^N)$, i.e., $\phi_*=m$ satisfies the Euler-Lagrange equation
$$
\Delta^2\phi_*+2\alpha \Delta \phi_*+\alpha^2\phi_*+\frac{d W}{d \phi}(\phi_*)=\alpha^2 m+\int_{\mathbb{T}^N} \frac{d W}{d \phi}(\phi_*)\, dx.
$$
Then we compute the second variation of $\cal F$ at $\phi_*$ over $H_m^2(\mathbb{T}^N)$: for every test configuration $u \in H^2(\mathbb{T}^N)$ with $\int_{\mathbb{T}^N} u\, dx =0$,
 \begin{align*}
 \nabla^2 {\cal F}(\phi_*)(u, u) &= \frac{d^2}{dt^2}\bigg|_{t=0}{\cal F}(\phi_*+tu)\\
 &= \int_{\mathbb{T}^N} (\alpha u+\Delta u)^2 +\frac{d^2 W}{d \phi^2}(\phi_*)u^2\, dx.
 \end{align*}
 By Step 1, we deduce that
 $$\nabla^2 {\cal F}(\phi_*)(u, u) \geq \left(\min_{k\in 2\pi \Z^N, k\neq 0} (\alpha-|k|^2)^2+\frac{d^2 W}{d \phi^2}(m)\right) \int_{\mathbb{T}^N} u^2\, dx.$$
 Therefore, if \eqref{cond_min_N} holds true, then $\phi_*=m$ is a stable point of $\cal F$ over $H_m^2(\mathbb{T}^N)$. Conversely, if 
  \eqref{cond_min_N} fails to be true, set $k_0 \in 2\pi \Z^N\setminus\{0\}$ be a minimum of $\min_{k\in 2\pi \Z^N, k\neq 0} (\alpha- |k|^2)^2$ and
choosing the test function $u(x)=\sin (k_0\cdot x)$, we obtain that
 $$
 \nabla^2 {\cal F}(\phi_*)(u, u)= \frac{1}{2}(\alpha- |k_0|^2)^2+\frac{1}{2} \frac{d^2 W}{d \phi^2}(m)< 0,
 $$
which proves the instability of $\phi_*$.

\medskip

\nd {\it Step 3. If  $W\in C^4(\R)$ satisfies $\frac{d^4 W}{d \phi^4}\geq w^2$ in $\R$ for some $w>0$, then for every $\phi\in H^2_m(\T)$,
\begin{align}
\label{bound_pot}
\int_{\mathbb{T}^N} W(\phi)\, dx&\geq W(m)+\frac12\frac{d^2 W}{d\phi^2}(m)\int_{\T}(\phi(x)-m)^2\, dx\\
\nonumber&\hspace{2cm}+\frac16\frac{d^3W}{d\phi^3}(m)\int_{\T}(\phi(x)-m)^3\, dx+\frac{w^2}{24} \int_{\T}(\phi(x)-m)^4\, dx.
\end{align}
}  Indeed,
since $\phi-m$ has vanishing average,
the Taylor expansion of $W$ in $m$ leads to
\begin{align*}
\int_{\mathbb{T}^N} \big(W(\phi) -W(m)\big)\, dx& =\frac12\frac{d^2 W}{d\phi^2}(m)\int_{\T}(\phi(x)-m)^2\, dx
+\frac16\frac{d^3W}{d\phi^3}(m)\int_{\T}(\phi(x)-m)^3\, dx\\
&\quad + \int_{\mathbb{T}^N} \int_0^1 \frac{(1-\ell)^3}6 \frac{d^4W}{d\phi^4}(m+\ell(\phi(x)-m))(\phi(x)-m)^4 \, d\ell dx;
\end{align*}
then \eqref{bound_pot} follows due to $\frac{d^4W}{d\phi^4}\geq w^2$.

\medskip

\nd {\it Step 4. Proof of 2.} If $\phi\in H^2_m(\T)$, we denote by $u=\phi-m$ of vanishing average.
Then Steps 1 and 3 yield
\begin{align}
\label{ineq_11}
&{\cal F}(\phi)-{\cal F}(m)\geq A+B+C \quad \textrm{ with }\\
\nonumber
&A=\frac12 \int_{\mathbb{T}^N} (\alpha u+\Delta u)^2+ \frac{d^2 W}{d \phi^2}(m) u^2\, dx, \,\, B=\frac16\frac{d^3W}{d\phi^3}(m)
\int_{\T}u^3\, dx, \, \, C=\frac{w^2}{24}\int_{\mathbb{T}^N}u^4 \, dx.
\end{align}
Note that by \eqref{interm1} and \eqref{cond_min_N}, we have that 
$$A\geq \frac14 \bigg(\frac{d^2 W}{d \phi^2}(m)+\min_{k\in 2\pi \Z^N, k\neq 0} (\alpha-|k|^2)^2\bigg) \int_{\T} u^2\, dx\geq 0.$$
We distinguish two cases:

$\bullet$ Case 1: $B=0$. By \eqref{ineq_11}, ${\cal F}(\phi)-{\cal F}(m)\geq A+C\geq 0$. In particular, we deduce that $\phi_*=m$ minimizes $\cal F$ over the set of functions  $\phi\in H_m^2(\mathbb{T}^N)$ with $\int_{\T}(\phi-m)^3\, dx=0$. Moreover, if $\phi$ is another minimizer in this class, then the above inequalities become equalities; in particular, $A=C=0$ yielding $u=0$, i.e., $\phi=m$ (because $w>0$). This yields the uniqueness of the minimizer $\phi_*=m$ over all functions $\phi\in H^2_m(\T)$ with $(\phi-m)^3$ of zero average.

$\bullet$ Case 2: $B\neq 0$. Then $\int_{\T}u^3\, dx\neq 0$ yielding by \eqref{cn} and \eqref{cond_optima}:
$$4AC\geq \frac{w^2P_N}{12}\left(\int_{\T}u^3\, dx\right)^2=3w^2 P_N B^2 \left(\frac{d^3W}{d\phi^3}(m)\right)^{-2}\geq B^2.$$
As $C>0$, it follows that $$A+B+C\geq \min_{t\in \R} \, (A+Bt+Ct^2)\geq 0.$$ We conclude by \eqref{ineq_11} that ${\cal F}(\phi)\geq {\cal F}(m)$ which implies 
that $\phi_*=m$ is a global minimizer of $\cal F$ over $H_m^2(\mathbb{T}^N)$.
Moreover, if the inequality in \eqref{cond_optima} is strict, we deduce that $4AC>B^2$, in particular, $A+B+C>0$; therefore ${\cal F}(\phi)> {\cal F}(m)$ yielding the uniqueness of the global minimizer.

 \medskip

\nd {\it Step 5. Proof of point 3.} By the assumptions at point 3. combined with Proposition \ref{prop}, we know that the infimum $P_N$ in \eqref{cn} is achieved by some function $u$ of zero average with $\int_{\T}u^3\, dx=1$. Within the notations at Step 4, we have for this minimizer $u$ in \eqref{cn}:
$$4AC=\frac{w^2P_N}{12}, \quad B^2=\frac1{36} \left(\frac{d^3W}{d\phi^3}(m)\right)^{2}.$$ As \eqref{cond_optima} fails to be true, i.e.,
$B^2>4AC$, there exists $t\in \R\setminus\{0\}$ such that $A+Bt+Ct^2<0$. Set $\phi=m+tu$. As $\frac{d^4 W}{d \phi^4}=w^2$ in $\R$, we have the equality in \eqref{ineq_11} and thus 
$${\cal F}(\phi)-{\cal F}(m)=t^2(A+Bt+Ct^2)<0,$$
which proves that $\phi_*=m$ is not a global minimizer of $\cal F$ over $H_m^2(\mathbb{T}^N)$.
\end{proof}

\begin{Remark}
\label{potPFC}
Let $W(\f)=\frac14 (\f^2-a)^2$ be the double-well potential used in the PFC model with $a>0$ and fix $\alpha=1$ (in particular, \eqref{cond:alpha} holds true).
Then we can apply Theorems~\ref{TheoremGammaLimit} and \ref{the:optimal} with
the conditions \eqref{cond_min} and \eqref{cond_optima} writing as 
\beq
\label{PFC_curve}
3m^2+(1-4\pi^2)^2\geq a \quad \textrm{and} \quad P_{N=2}\geq 2m^2,
\eeq
where $P_2$ depends on $\frac{d^2 W}{d \f^2}(m)=3m^2-a$. 
The above system determines the so-called order/disorder transition curve separating in the plane $(m, a)$ the region where the uniform state is optimal.
Note that the curve found numerically in \cite{EKHG} has the same aspect as the above parabola. In \cite{SCN},
the sufficient condition $a \leq m^2$ was found analytically which is a subregion in our result because by Proposition \ref{prop} we proved that 
$$P_2\geq    \frac{d^2 W}{d \phi^2}(m)+\min_{k\in 2\pi \Z^2, k\neq 0} (\alpha-|k|^2)^2=3m^2-a+(1-4\pi^2)^2\geq 2m^2,$$
whenever $a \leq m^2$. 
As our condition \eqref{PFC_curve} is necessary and sufficient,
we conclude that this is the exact region separating the regime of trivial minimizers from non-trivial ones.
\end{Remark}

\begin{Remark}
\label{rem:sym}
A challenging question is to determine the curve separating the parameter region where every global minimizer of $\cal F$ over $H_m^2(\mathbb{T}^N)$ is one-dimensional 
(that corresponds in particular to the region where stripes structures nucleate in the system, see e.g. \cite{EKHG, SCN}).
(This question is related to the well-known conjecture of De Giorgi for minimal surfaces.) Very few analytical results are available:
we mention in particular the result in \cite{BH} for the one-dimensional symmetry in the extended Fisher-Kolmogorov model in $\R^N$.
Also, the results in \cite{IgMo} for the one-dimensional symmetry in the Aviles-Giga type models in $\R^N$ (recall that in $2$-dimensions,
the standard Aviles-Giga model can be seen as a forth order problem in the stream function corresponding to the order parameter, see \cite{Ambrosio:1999, Aviles:1987, Aviles:1999, Jin:2000}).
\end{Remark}

\subsection{The case of the Ohta-Kawasaki model.}

\begin{proof}[Proof of Proposition \ref{prop2}]
In terms of the Fourier representation of a function $u \in H^1({\mathbb{T}^N})$ of zero average, i.e.,
$$
u(x)=\sum_{k\in 2\pi \Z^N, k\neq 0} \big(a_{k}\cos(k\cdot x))+b_{k}\sin(k\cdot x)\big), \quad x\in \mathbb{T}^N,
$$
where $a_k, b_k\in \R$ for  $k\in 2\pi \Z^N\setminus\{0\}$,
we write
\beq
\label{starOK}
\int_{\mathbb{T}^N}  \frac1{\gamma^2}|\nabla u|^2 + |\nabla(-\Delta)^{-1}u|^2 \,dx= \frac{1}{2}
\sum_{k\in 2\pi \Z^N\setminus\{0\}} \left(\frac{|k|^2}{\gamma^2} +\frac1{|k|^2}\right) (a_{k}^2+b_{k}^2).
\eeq
Therefore,
\begin{align*}
&\int_{\mathbb{T}^N}  \frac1{\gamma^2}|\nabla u|^2 + |\nabla(-\Delta)^{-1}u|^2 +\frac{d^2 W}{d\phi^2}(m)u^2 \,dx\\
&\quad \quad \geq  \left(\frac{d^2 W}{d\phi^2}(m)+ \min_{k\in 2\pi \Z^N\setminus\{0\}} \big(\frac{|k|^2}{\gamma^2} +\frac1{|k|^2}\big)\right) \int_{\mathbb{T}^N} u^2
\end{align*}
which is nonnegative thanks to \eqref{cond_OK}.  
The conclusion follows by the same argument as in the proof of Proposition \ref{prop}. The only difference consists in the fact that minimizing sequences in \eqref{qn} are bounded in $H^1$ (instead of $H^2$ as in the case of PFC model); therefore, we need the compact embedding $H^1(\T)\subset L^p(\T)$ for $p\in [1,3]$ provided that $N<6$ and one also uses the compact embedding $H^1(\T)\subset \dot{H}^{-1}(\T)$.
\end{proof}

\begin{proof}[Proof of Theorem \ref{the:optimal_OK}] We start by noting that $\phi_*=m$ is a critical point of the Ohta-Kawasaki functional $\cal E$ over $H_m^1(\mathbb{T}^N)$,
i.e., $\phi_*$ satisfies the Euler-Lagrange equation
$$
-\frac1{\gamma^2}\Delta \phi_*+ \psi_*+\frac{d W}{d \phi}(\phi_*)=\int_{\mathbb{T}^N} \frac{d W}{d \phi}(\phi_*)\, dx,
$$
where $\psi_*$ is the solution of \eqref{psi} associated to the critical point $\phi_*$ (obviously, $\psi_*=0$ if $\phi_*=m$).
The second variation of $\cal E$ at a critical point $\phi_*$ is given for every test configuration $u \in H^1(\mathbb{T}^N)$ with $\int_{\mathbb{T}^N} u \, dx =0$:
 \begin{align*}
 \nabla^2 {\cal E}(\phi_*)(u, u) &= \frac{d^2}{dt^2}\bigg|_{t=0}{\cal E}(\phi_*+tu)\\
 &= \int_{\mathbb{T}^N} \frac1{\gamma^2}|\nabla u|^2 + |\nabla(-\Delta)^{-1}u|^2 +\frac{d^2 W}{d \phi^2}(\phi_*)u^2\, dx.
 \end{align*}
By \eqref{starOK}, the conclusion of point 1. follows. For points 2. and 3., if $\phi \in H_m^1({\mathbb{T}^N})$, we write the Fourier representation
$$
\phi(x)=m + \sum_{k\in 2\pi \Z^N, k\neq 0} \big(a_{k}\cos(k\cdot x))+b_{k}\sin(k\cdot x)\big), \quad x\in \mathbb{T}^N,
$$
where $a_k, b_k\in \R$ for  $k\in 2\pi \Z^N\setminus\{0\}$. Denoting $u=\phi-m$ of vanishing average, by \eqref{bound_pot}, we obtain that
\begin{align*}
{\cal E}(\phi)-{\cal E}(m)&\geq \tilde A+B+C \quad \textrm{ with }\\
& \tilde A=\frac12 \int_{\mathbb{T}^N} \frac1{\gamma^2}|\nabla u|^2 + |\nabla(-\Delta)^{-1}u|^2+ \frac{d^2 W}{d \phi^2}(m) u^2\, dx,\end{align*}
$B$ and $C$ being the same as in \eqref{ineq_11}. 
The conclusion of points 2. and 3. follows by the same argument as in the proof of Theorem \ref{gen:optimal}.
\end{proof}

\paragraph{Acknowledgment.}  The authors thank Xavier Lamy for useful comments. R.I. acknowledges partial support by the ANR project ANR-14-CE25-0009-01.

\end{document}